\documentclass[11pt,letter]{article}
\usepackage{7400-style}

\usepackage{etoolbox}
\usepackage{tikz}
\usetikzlibrary{positioning}
\usetikzlibrary{calc}

\newcommand{\pprime}{\raisebox{0pt}[0pt][0pt]{${}^\prime$}}  

\tikzstyle{u}=[line width=4pt,blue,line cap=butt]

\usepackage[margin=0.8in]{geometry}
\usepackage{fancyhdr}

\graphicspath{ {images/} }

\newcommand{\la}{\langle}
\newcommand{\ra}{\rangle}
\newcommand{\Ho}{\mathcal{H}}
\newcommand{\Ro}{\mathcal{R}}
\renewcommand{\cC}{\mathcal{C}}
\renewcommand{\cF}{\mathcal{F}}

\title{Character varieties as a tensor product}
\author{Martin Kassabov, Sasha Patotski}

\begin{document}
\maketitle

\begin{abstract}
  In this short note we show that representation and character varieties of discrete groups can be viewed as tensor products of suitable
  functors over the PROP of cocommutative Hopf algebras.
  Such view point has several interesting applications. First, it gives a straightforward way of
  deriving the functor sending a discrete group to the functions on its representation variety,
  which leads to representation homology. Second, using a suitable deformation of the functors involved in this
  construction, one can obtain deformations of
  the representation and character varieties for the fundamental groups of $3$-manifolds, and could lead to better understanding of quantum representations of mapping class groups.
\end{abstract}

\section{Introduction}

    Generally speaking, representation variety (more precisely, scheme) is the scheme parametrizing
    homomorphisms of one algebraic object into another one. For example, there are representation schemes $\rep(A,V)$ parametrizing
    homomorphisms of associative algebras $A\to \End(V)$; or the scheme $\rep(\Gamma,G)$ parametrizing group homomorphisms
    from a discrete group $\Gamma$ to an algebraic group $G$; or the scheme $\rep(\mathfrak{a},\mathfrak{g})$
    of Lie algebra homomorphisms $\mathfrak{a}\to \mathfrak{g}$. In each of the examples, there is a natural group
    acting on these schemes, and the character variety is the (categorical) quotient of the representation variety
    by the action of the group. For example, $\mathrm{GL}(V)$ acts on $\rep(A,V)$ by conjugation, $G$ acts on $\rep(\Gamma,G)$,
    and $\rep(\mathfrak{a},\mathfrak{g})$ has the natural action of an algebraic group $G$ with $\mathfrak{g}=\cL ie(G)$.
    The geometry of representation and character varieties has been heavily studied (see, for example, \cite{Kr82},\cite{Sik12} and references therein). Character varieties play an important role in knot theory (\cite{FG},\cite{PS00},\cite{CM12} to name a few references), moduli of connections (see, for example, \cite{AB83},\cite{Gold84}), they give interesting invariants called representation homology (see \cite{BKR}, and Section~\ref{sec:remark}).

    In this paper we are trying to advocate a slightly different point of view on constructing representation and character varieties.
    Let $\Ho$ be the PROP of cocommutative Hopf agebras, that is, a small category with $\ob(\Ho)=\dN$ and with $\operatorname{Mor}(\Ho)$
    generated by certain morphisms, suggestively denoted by $m,\Delta,\eta,\varepsilon,S,\tau$. For a precise definition,
    see Section \ref{sec:prop-of-hopf-alg}. There is a natural correspondence 
    between strict monoidal functors $\Ho\to \cV ect$ and cocommutative Hopf algebras.
    For any discrete group $\Gamma$, there is a natural cocommutative Hopf algebra associated to it, namely, its group algebra $k[\Gamma]$.
    The group algebra, in turn, determines a functor $k[\Gamma]\colon \Ho\to \cV ect$ sending $[n]$ to $k[\Gamma]^{\otimes n}$.
    Similarly, the algebra of regular functions $k(G)$ on an affine algebraic group $G$ is a \emph{commutative}
    Hopf algebra, and therefore determines a functor $k(G)\colon \Ho^{op}\to \cV ect$.
    The following is one of the main results of this note.

	\begin{reptheorem}{thm:rep-var-groups-as-tensor-prod}
		There are natural isomorphisms of commutative algebras
		\begin{enumerate}
			\item $[k(G)]\otimes_\Ho k[\Gamma]\simeq k(\rep(\Gamma,G))$;
			\item if $\Gamma_1,\Gamma_2$ are two discrete groups, then
			$\hom^\Ho(k[\Gamma_1],k[\Gamma_2])\simeq k[\rep(\Gamma_1,\Gamma_2)]$, where $k[\rep(\Gamma_1,\Gamma_2)]$
			denotes the algebra of functions with finite support on the discrete set $\rep(\Gamma_1,\Gamma_2)$.
		\end{enumerate}
	\end{reptheorem}
	
	In the theorem, $[k(G)]\otimes_\Ho k[\Gamma]$ denotes the tensor product of two functors, defined by the formula \eqref{eq:tensor}.
	By $\hom^\Ho$ we denote the space of all ($k$-linear) natural transformations between two functors.
	
    It is important to note that this theorem is not really useful for computations. Indeed, unraveling the definition of the tensor product over $\Ho$ one is quickly lead to a quotient of
    the space of functions on several copies of $G$ by the relations coming from a presentation of the group $\Gamma$.

    The result essentially equivalent to the first part of this Theorem was independently obtained by Massuyeau--Turaev, see~\cite{MT}. They used that both functors $[k(G)]$ and $k[\Gamma]$ are monoidal,
    which turns the tensor product over $\Ho$ into a quotient of the tensor algebra of $k(G) \otimes k[\Gamma]$.
    One of the advantages of our viewpoint is that
    it can be applied even to non-monoidal functors. In particular, this construction also allows to view the \emph{character}
    variety as the tensor product over $\Ho$ of the functor $k[\Gamma]$
	and the functor $[k(G)]^G\colon \Ho^{op}\to \cV ect$ sending $[n]$ to $[k(G^n)]^G$, the algebra of invariants
	under the diagonal action of $G$ on $G^n$ by conjugation. Notice that since it is not a strict monoidal functor,
	it does not correspond to a Hopf algebra.
	
	Another advantage of viewing the character variety as a tensor product of two functors over the category $\Ho$
	is that it gives an alternative way of defining representation homology for Lie algebras and groups,
	see Section~\ref{sec:remark} for more details.
	
	We feel that the main advantage is the fact that it allows
	to quantize the character variety of a surface group in such a way that the action of $\cM_{\Sigma}$ would be automatically present.
	We sketch this in Section~\ref{sec:deform} and provide details in the forthcoming paper~\cite{KP}. The idea is to deform
	each of the three ingredients in $k(G)\otimes_\Ho [\pi_1(\Sigma)]$. Namely, we replace the commutative algebra $k(G)$
	by the algebra $k_q(G)$ of functions on the corresponding quantum group. The fundamental group $\pi_1(\Sigma)$
	is replaced by the set $[\Sigma]$ of isotopy classes of ribbons inside $\Sigma\times [0,1]$ with ends in a small
	disk near the basepoint of $\Sigma$. Finally, we replace $\Ho$ by a certain category $\Ro$ analogous
	to the category of ribbons. Then, the tensor product $k_q(G)\otimes_{\Ro} [\Sigma]$ will again be an algebra
	(but no longer commutative), which will be a deformation of $k(G)\otimes_\Ho [\pi_1(\Sigma)]\simeq \rep(\pi_1(\Sigma),G)$.
    Such construction of a deformation of the representation variety is similar to the one announced in~\cite{Hab}.

    The ultimate goal of the viewpoint explained in this note and of the aforementioned construction in~\cite{KP}
    is to obtain a better understanding of the quantum representations of the mapping class
    group $\cM_g$ of a closed surface of genus $g$. There are several different constructions
    of these representations: for the approach via TQFT and skein theory see, for example, \cite{Fun99}, \cite{BHMV}, and the book
    \cite{Koh02} for a nice exposition. For a short summary of the geometric quantization approach see~\cite{Mas03}
    and references therein. It was recently proved by Andersen and Ueno in a sequence of papers
    (see~\cite{AU1}, \cite{AU2}, \cite{AU3} and~\cite{AU4}) that both constructions give isomorphic representations of the mapping class group.
    In~\cite{KP} we show that the skein theoretic construction of the quantum representations can be obtained as tensor product where one of the functors is related to the small quantum group at roots of unity.
    We hope that the close relationship between such tensor product and the representation varieties can be used to show that the construction of quantum representations via geometric quantization
    can also be obtained using similar tensor product. Ideally, this will lead to semi-elementary proofs
    of many known properties of quantum representations, like (in)finiteness properties (see \cite{Gil99}, \cite{Mas99})
    or asymptotic faithfulness (see~\cite{And06} and~\cite{FWW}).
	
	\medskip
	
	\textbf{Acknowledgements}: We thank  Yuri Berest and Jim Conant for a lot of helpful discussions and comments. Martin Kassabov was partially supported by NSF grants DMS 1303117 and 1601406.

\section{PROP of cocommutative Hopf algebras}
\label{sec:prop-of-hopf-alg}
\subsection{The category $\Ho$}
	Let $\Ho$ be the \emph{PROP of cocommutative Hopf algebras} ---
	this is a small category with objects $\ob(\Ho)=\dN$, denoted by $[n],\, n\in\dN$.
    The category $\Ho$ has an additional monoidal structure is given by $[n]\otimes [m]=[n+m]$.
	The set of morphisms is generated by the morphisms $\mu\colon [2]\to [1]$, $\Delta\colon [1]\to [2]$, $S\colon [1]\to [1]$,
	$\eta\colon [0]\to [1]$, $\varepsilon\colon [1]\to [0]$, $\tau\colon [2]\to [2]$ subject
	to the relations satisfied by cocommutative Hopf algebras. For example,
	the associativity is the equality of the following morphisms
    $\mu \circ (\mu \otimes \id) = \mu \circ (\id \otimes \mu)$ as elements in $\hom_{\Ho}([3],[1])$.
    See Section~\ref{subsec:axioms-Hopf-alg} for the complete list of relations.
    We will sometimes consider the linearized version of $\Ho$, where the space of morphisms from $[n]$ to $[m]$
    will be the vector space spanned by the set $\hom_\Ho([n],[m])$.
    We will abuse the notation and denote this category by $\Ho$.

	\begin{remark}
		\label{rem:functors-for-hopf-alg}
		Immediately from the definition, there is an equivalence between the category
		of cocommutative Hopf algebras and the category of \emph{strict monoidal} functors $\Ho\to \cV ect$,
		assigning to a Hopf algebra $H$ the functor $[H]\colon [n]\mapsto H^{\otimes n}$.
		
		Similarly, the category of \emph{commutative} Hopf algebras is equivalent to the category of
		\emph{contravariant} strict monoidal functors $\Ho^{op}\to \cV ect$.

	\end{remark}

	The connection between category $\Ho$ and group theory arises from the following folklore statement:

	\begin{theorem}
        \label{thm:folklore}
		There is an equivalence of categories $\cH\simeq \cG r_{ff}^{op}$,
		where $\cG r_{ff}$ is the category of finitely generated free groups.
	\end{theorem}

    The few references we found are~\cite[Section 5]{Pir02} and \cite[Section 2]{Wh}, which do not contain a proof.
    A detailed (and semi-trivial) proof of the following Corollary (which is equivalent to Theorem~\ref{thm:folklore})
    can be found in~\cite[Thm. 4.5]{CK}.
	
	\begin{corollary}
		\label{cor:homs-in-H}
		For any integers $n,m$ we have $\hom_\Ho([n],[m]) \simeq F_n^{\times m}$,
		where $F_n$ denotes the free group on $n$-generators (which comes with a specified choice of generators).
	\end{corollary}
	
	\begin{corollary}
		\label{cor:functor-F-n-is-corepres}
		The functor $h^{[n]}=\hom_\Ho([n],-)$ corepresented by $[n]\in\ob(\Ho)$ is isomorphic to the functor
		$[F_n]\colon \Ho\to \cS ets$ sending $[m]\mapsto F_n^{\times m}$.
	\end{corollary}

	Let $F\colon \Ho\to \cV ect$ and $G\colon \Ho^{op}\to \cV ect$ be any two functors.
	We define their tensor product $G\otimes_\Ho F$ to be the vector space defined by
	\begin{equation}
		\label{eq:tensor}
	G\otimes_\Ho F = \bigoplus\limits_{n\geqslant 0} G([n])\otimes F([n])\Big/\left( f^\ast(v)\otimes w\sim v\otimes f_\ast(w) \right)
	\end{equation}
	for all $f\in \hom_\Ho([m],[n])$, $v\in G([n])$ and $w\in F([m])$. This construction is also sometimes called the
	\emph{coend} of two functors, see \cite[Section IX.6]{ML}. This definition can be obviously modified to the case
	when both functors are valued in $\cS ets$ instead of $\cV ect$. If $F\colon \Ho\to \cS ets$ and $G\colon \Ho^{op}\to \cV ect$,
	the product $G\otimes_\Ho F$ denotes the tensor product of $G$ and $kF$, $kF([n]):=\operatorname{Span}_k(F[n])$.
	
	\begin{proposition}
		\label{prop:tensor-for-represen-functors}
		For any functor $F\colon \Ho\to \cS ets$ (not necessarily monoidal) there is a natural bijection
		\[
		\hom^\Ho([F_n],F)\simeq F([n])
		\]
		
		If $G\colon \Ho^{op}\to \cS ets$ is a contravariant functor, then
		\[
		G \otimes_{\Ho} [F_n] \simeq G([n])
		\]
	\end{proposition}
	\begin{proof}
		The first part is just Yoneda lemma. The second part is an application of Lemma \ref{lem:tensor-with-corepres-functor}.
	\end{proof}

    \begin{remark}
        If $F$ and $F'$ are two monoidal functors from $\Ho\to \cS ets$, there are two ways to interpret the set $\hom^\Ho(F,F')$. This can denote either the
        set of all natural transformations $\{\phi_{[n]}\colon F([n]) \to F'([n]) \}$ between the functors,
        without requiring any coherence with the monoidal structures on $\Ho$ and $\cS ets$.
        Alternatively, one can consider the set of natural transformations preserving the monoidal structure,
        which is equivalent to adding the
        additional condition $ \phi_{[n+ n']} = \phi_{[n]}\times \phi_{[n']}$.
        We stress that we use the first viewpoint and \emph{do not require} $\phi_{[n]}$ to be compatible with the monoidal structure.
    \end{remark}

    \subsection{Examples of functors from $\Ho$}

    Our main examples of covariant functors $\Ho\to \cV ect$ come from discrete groups and Lie algebras.
    For any discrete group $\Gamma$ there is a natural functor $k[\Gamma]\colon \Ho \to \cV ect$
    associated to the group algebra given by $[n]\mapsto k[\Gamma]^{\otimes n}$. The action of $\Gamma$
    on itself by conjugation induces a natural action of $\Gamma$ on the functor $k[\Gamma]$.
    Taking invariants under this action gives a functor $k[\Gamma]^\Gamma\colon \Ho\to \cV ect$
    sending $[n]$ to $k[\Gamma^{\times n}]^\Gamma$. Note that this functor is not strictly monoidal,
    but only lax monoidal. In particular, it does not correspond to any cocommutative Hopf algebra.
	
	Similarly, if $\mathfrak{a}$ is a Lie algebra, there is a strict monoidal functor $[\cU\mathfrak{a}]\colon \Ho \to \cV ect$
	associated to it, sending $[n]$ to $\cU\mathfrak{a}^{\otimes n}$.
	If $A$ is a Lie group with $\cL ie(A)=\mathfrak{a}$, there are natural functors $[\cU \mathfrak{a}]^{\mathfrak{a}}$
	and $[\cU \mathfrak{a}]^{A}$ sending $[n]$ to the invariants $\left(\cU \mathfrak{a}^{\otimes n}\right)^{\mathfrak{a}}$ and
	$\left( \cU \mathfrak{a}^{\otimes n} \right)^A$ of the diagonal action of $\mathfrak{a}$ and $A$, respectively.
	The latter two functors are only lax monoidal.
	
	Our main examples of contravariant functors $\Ho^{op}\to \cV ect$ come from affine algebraic groups.
	If $G$ is an affine algebraic group, its algebra of regular functions $k(G)$
	defines a contravariant strict monoidal functor which we denote simply by $[k(G)]\colon \Ho^{op}\to \cV ect$
	given by $[n]\mapsto k(G)^{\otimes n}\simeq k(G^n)$.
	Once again, taking invariants of the adjoint action gives a lax monoidal functor $[k(G)]^G\colon [n]\mapsto k(G^n)^G$.
	
    \section{Some properties of the tensor product}

	Let $F\colon \cC\to \cD$ be a functor and let $\varphi$ be morphism in
	$\hom_{\cD} ( d, F(c))$ for some objects $c \in \ob\cC$ and $d\in \ob\cD$.
	Then $\varphi$ can be extended to a natural transformation $d \cdot \hom_{\cC}(c,-) \Rightarrow F$,
	\[
	\coprod\limits_{c\to c'}d\to \coprod\limits_{c\to c'}F(c)\to F(c')
	\]
	where coproducts are taken over the set $\hom_\cC(c,c')$.
	
	Let $\{\varphi_i\}$ a collection of morphisms in $\hom_{\mathcal{D}} ( d, F(c))$. Denote by
	$F/\{\varphi_i\}$ the coequalizer in the functor category $\cF un(\mathcal{C}, \mathcal{D})$ of the natural transformations $d \cdot \hom_{\mathcal{C}}(c,-)\Rightarrow F$ induced by $\varphi_i$.	
	
	\begin{lemma}
		\label{lem:quot-functor}
		For any object $c'\in\ob(\cC)$, $F/\{\varphi_i\}(c')$ is the coequalizer in $\mathcal{D}$ of the maps $\widetilde \varphi_i: \hom_{\mathcal{C}}(c,c')\times d \to F(c')$ defined by
		$\widetilde \varphi_i(\psi) = F_{\psi}\circ \varphi_i$.
	\end{lemma}
	\begin{proof}
		This follows immediately from Lemma \ref{lem:colims-in-functors}.
	\end{proof}

    	Let $\Gamma = \la g_1,\dots,g_n \mid r_0=e, r_1,\dots, r_m,\dots \ra$ be a finitely generated group.
	By Corollary \ref{cor:homs-in-H}, we can identify $\hom_\Ho([n],[n+1])$ with $F_n^{\times n+1}$.
	Let $\varphi_i\in \hom_\Ho([n],[n+1])$ be the morphism corresponding to the element $(g_1,\dots, g_n, r_i)\in F_n^{\times n+1}$.
	
	\begin{lemma}
		\label{lem:group-funct-as-quot}
		In the notation above, the functor $[\Gamma]\colon \Ho\to \cS ets$ is isomorphic to $h^{[n]}/\{\varphi_i\}$,
		where $h^{[n]}=\hom_\Ho([n],-)$.
	\end{lemma}
	\begin{proof}
		For any $[m]\in\ob(\Ho)$ there is a natural surjection of sets $h^{[n]}([m])\simeq F_n^m\onto [\Gamma]([m])=\Gamma^m$.
		We need to prove that it factors through the coequalizer $h^{[n]}/\{\varphi_i\}([m])$, and the
		induced map is actually an isomorphism.
		
		For notational simplicity, let's only consider the case $m=1$, the general case being completely analogous.
		Unraveling the definition of $h^{[n]}/\{\varphi_i\}$,
		we need to prove that the coequalizer of the maps
		\[
			\coprod\limits_{\omega \in F_{n+1}} F_n^n \stackrel{\varphi_i}{\longrightarrow} F_n
		\]
		is isomorphic to $\Gamma$, where on the component corresponding to a word
		$\omega= \omega(x_1,\dots,x_{n+1})\in F_{n+1}=\langle x_1,\dots,x_{n+1} \rangle$
		the map $\varphi_i$ sends a tuple of words $(y_1,\dots,y_n)$ with $y_i(g_1,\dots,g_n)\in F_n$ to the word
		\[
		\omega(y_1(g_1,\dots,g_n),\dots,y_n(g_1,\dots,g_n),r_i(g_1,\dots,g_n))\in F_n
		\]
		The coequalizer is the quotient of $F_n$ identifying all the words above. The set of such words
		is the normal subgroup of $F_n$ (viewed as a set) generated by the relations $\{r_i\}$.
		The identification means that we set identify every element of the normal subgroup with identity,
		which is the description of $\Gamma$.
	\end{proof}

		\begin{lemma}
			\label{lem:TV-functor}
			Let $V$ be a finite-dimensional vector space. There is an isomorphism of functors from $\Ho$ to $\cV ect$:
			\[
			[TV]\simeq \bigoplus\limits_{n=0}^{\infty} V^{\otimes n}\otimes_{S_n} \left( h^{[n]}\Big/ \{\varphi_i,\psi_j, \nu_j \} \right)
			\]
			where the morphisms $\varphi_i,\psi_j,\nu_j\colon k\to \hom_\Ho([n],[n+1]),\, i\geq 0$ are given by $\varphi_0=0$ and
			\begin{align*}
				\varphi_i &=\id^{\otimes i-1}\otimes \Delta\otimes\id^{\otimes n-i}-\id^{\otimes i-1}\otimes \eta \otimes \otimes\id^{n-i+1}-
				\id^{\otimes i}\otimes \eta \otimes \id^{n-i}\\
				\psi_i &= \id^{\otimes i-1}\otimes S\otimes\id^{\otimes n-i} + \id^{\otimes n} \\
                \nu_i &= \id^{\otimes i-1}\otimes \varepsilon \otimes\id^{\otimes n-i}
			\end{align*}
		\end{lemma}
	\begin{proof}

		For simplicity, let's only prove $[TV]([m])\simeq \bigoplus V^{\otimes n}\otimes_{S_n} \left( h^{[n]}/ \{\varphi_i \}\right)([m])$
		in the case $m=1$. In this case, $[TV]([1])=TV=\bigoplus_n V^{\otimes n}\simeq \bigoplus_n V^{\otimes n}\otimes_{S_n} k[S_n]$.
		Therefore, it is enough to show that $k[S_n]\simeq h^{[n]}/ \{\varphi_i \}([1])$. 
		We emphasize that this is only isomorphism between vector spaces, since the right side does not have natural algebra structure.

		Any morphism $\alpha\colon [n]\to [1]$ in $\Ho$ can be represented as a composition of a certain number of comultiplications,
		followed by antipodes, then a permutation, and finally followed by the multiplication.
		
		The relations $\varphi_i\sim 0$ allow us to remove all comultiplications, and after that all antipodes. Thus, a morphism in $\hom_\Ho([n],[1])/\{\varphi_i \}$
        is represented (up to sign) by a permutation, followed by the multiplication, i.e. by an element of $k[S_n]$. It is easy to see that this is an isomorphism.
	\end{proof}
	
	\begin{proposition}
			\label{prop:alg-structure-on-tensor-vect}
			Let $F\colon \Ho\to \cV ect$ and $G\colon \Ho^{op}\to \cV ect$ be two lax monoidal functors.
			Then $G\otimes_\Ho F$ has a natural structure of a commutative algebra.
	\end{proposition}
	\begin{proof}
		This is a particular case of a more general fact \ref{prop:alg-structure-on-tensor}. The multiplication is
		induced by the lax monoidal structures of $F$ and $G$:
		\[
		\big(G([n])\otimes F([n]) \big) \otimes \big( G([m])\otimes F([m]) \big) \simeq \big( G([n])\otimes G([m])\big) \otimes \big( F([m])\otimes F([m])\big) \to G([m+n])\otimes F([m+n])
		\]
	\end{proof}
	
	\begin{lemma}
		For a lax monoidal contravariant functor $G\colon \Ho^{op}\to \cV ect$,
		the isomorphism
		$G\otimes_\Ho [F_n]\simeq G([n])$ of Proposition \ref{prop:tensor-for-represen-functors} is an isomorphism of algebras.
	\end{lemma}
	\begin{proof}
		Here $G(n)$ has a natural algebra structure induced by the morphism $\Delta^n\colon [n]\to [2n]$ in $\Ho$ and the lax structure on the functor $G$. Once again, this result follows from a more general one, see Lemma \ref{lem:isom-of-algebras-from-tensor}.
	\end{proof}

    \section{Representation and character varieties}

	Let $G$ be an affine algebraic group scheme, and denote its regular ring by $k(G)$.
	The ring $k(G)$ has a natural structure of a commutative Hopf algebra, and therefore defines
	a contraviariant functor from $\Ho^{op}\to \cV ect$, sending $[n]\in\ob(\Ho)$ to $k(G^{\times n})\simeq k(G)^{\otimes n}$.
	We denote this functor by $[k(G)]$.
	
	\begin{theorem} There are natural isomorphisms of commutative algebras
		\label{thm:rep-var-groups-as-tensor-prod}
		\begin{enumerate}
			\item $[k(G)]\otimes_\Ho k[\Gamma]\simeq k(\rep(\Gamma,G))$;
			\item if $\Gamma_1,\Gamma_2$ are two discrete groups, then
			$\hom^\Ho(k[\Gamma_1],k[\Gamma_2])\simeq k[\rep(\Gamma_1,\Gamma_2)]$, where $k[\rep(\Gamma_1,\Gamma_2)]$
			denotes the algebra of functions with finite support on the discrete set $\rep(\Gamma_1,\Gamma_2)$.
		\end{enumerate}
		
	\end{theorem}

	\begin{proof}
		
		The structure of a commutative algebra on the vector space $[k(G)]\otimes_\Ho [\Gamma]$ is given by
		Proposition \ref{prop:alg-structure-on-tensor-vect} since both functors $[k(G)]$ and $[\Gamma]$ are lax monoidal.
		
		If $\Gamma=F_n$ is a free group, the representation scheme is isomorphic to $\rep(F_n,G)\simeq G^n$,
		and so $k(\rep(F_n,G))\simeq k(G)^{\otimes n}$.
		Since $[F_n]=h^{[n]}$, Proposition \ref{prop:tensor-for-represen-functors} implies
		the isomorphism $k(G)\otimes_\Ho h^{[n]}\simeq k(G)^{\otimes n}$, which gives the result.
		
		If $\Gamma$ is not free, picking a presentation $\Gamma=\langle g_1,\dots,g_n\mid r_0=1,r_1,r_2,\dots \rangle$
		and using Lemma \ref{lem:group-funct-as-quot} identifies the functor $[\Gamma]$ with $\hom_\Ho([n],-)/\{\varphi_i\}$
		which by definition is a coequalizer $\colim [F_{n+1}\rightrightarrows F_n]$.
		
		The algebra $k(\rep(\Gamma,G))$ is defined by the adjunction
		\[
		\hom_{\cC om \cA lg}(k(\rep(\Gamma, G)),B)\simeq \hom_{\cG r}(\Gamma, G(B))
		\]
		Therefore, we need to prove that the commutative algebra $[k(G)]\otimes_\Ho [\Gamma]$ satisfies the adjunction above. We have
		\begin{align*}
			\hom_{\cC om \cA lg}\big([k(G)]\otimes_\Ho [\Gamma],B\big) &\simeq \hom_{\cC om \cA lg}\big([k(G)]\otimes_\Ho \colim \left([F_{n+1}\rightrightarrows F_n]\right),B\big)\\
			& \simeq \hom_{\cC om \cA lg}\big(\colim \left([k(G)]\otimes_\Ho [F_{n+1}\rightrightarrows F_n]\right),B\big)\\
			& \simeq \lim \left(\hom_{\cC om \cA lg}\big(k(G^{n}),B\big) \rightrightarrows  \hom_{\cC om \cA lg}\big(k(G^{n+1}),B\big)\right)\\
			& \simeq \lim \left(\hom_{\cG r}\big(F_{n},G(B)\big) \rightrightarrows  \hom_{\cG r}\big(F_{n+1},G(B)\big)\right)\\
			& \simeq \hom_{\cG r}\big(\colim (F_{n+1}\rightrightarrows F_n),G(B)\big)\\
			& \simeq \hom_{\cG r}\big(\Gamma,G(B)\big).
		\end{align*}
        In this computation most of isomorphisms follow from the commutativity between limits/colimits and hom functors. The only isomorphism which requires any explanation is between
        $$
        \hom_{\cC om \cA lg}\big(k(G^{n}),B\big) \rightrightarrows  \hom_{\cC om \cA lg}\big(k(G^{n+1}),B\big)
        \quad \mbox{and} \quad
        \hom_{\cG r}\big(F_{n},G(B)\big) \rightrightarrows  \hom_{\cG r}\big(F_{n+1},G(B)\big),
        $$
        which follows from the natural isomorphism between the sets $\hom_{\cC om \cA lg}\big(k(G^{n}),B\big)$
        and $\hom_{\cG r}\big(F_{n},G(B)\big)$
		with the $n$-tuples of elements in the group $G(B)$.

		The algebra structure on $k(G)\otimes_\Ho [\Gamma]$ is induced by the algebra structure on
		$k(G)\otimes_\Ho [F_n]$ (using the notation above).	
		The algebra $k(\rep(\Gamma,G))$ is a quotient of $k(\rep(F_n,G))\simeq k(G^{\times n})$. Thus,
		it is enough to check that the algebra structures coincide in the case of a free group $\Gamma=F_n$.
		
		But in this case, the functor $[F_n]$ is corepresentable, $[F_n]=h^{[n]}$ by Corollary \ref{cor:functor-F-n-is-corepres}.
		The object $[n]\in\ob(\Ho)$ is a counital coalgebra object in $\Ho$ with the structure maps $[n]\to [2n]$
		and $[0]\to [n]$ given by
		\begin{align*}
			\Delta^n & =\Delta\otimes\dots\otimes \Delta\colon  [n]\simeq [1]^{\otimes n}\to [2]^{\otimes n}\simeq [2n]\\
			\varepsilon^n & =\varepsilon\otimes\dots\otimes \varepsilon \colon  [0]\simeq [0]^{\otimes n}\to [1]^{\otimes n}\simeq [n]
		\end{align*}
		Therefore, we can apply Lemma \ref{lem:isom-of-algebras-from-tensor}, which gives the first isomorphism of the theorem.

		\vspace{3mm}
		
		The vector space $\hom^\Ho([\Gamma_1],[\Gamma_2])$ has an algebra structure induced by the convolution product.
		Namely, if $\alpha,\beta\colon [\Gamma_1]\to [\Gamma_2]$ are natural transformations,
		we define their product $\alpha\ast \beta$ by
		\[
		\xymatrix{
			[\Gamma_1][n] \ar@{-->}[d]_-{\alpha\ast \beta}  \ar[r]^-{[\Gamma_1][\Delta^n]} & [\Gamma_1][2n] \ar[r] & [\Gamma_1][n]\otimes [\Gamma_1][n] \ar[d]^-{\alpha[n]\otimes \beta[n]} \\
			[\Gamma_2][n] & [\Gamma_2][2n] \ar[l]_-{[\Gamma_2][m^n]} & [\Gamma_2][n]\otimes [\Gamma_2][n] \ar[l]
			}
		\]
		In the diagram, the functor $[\Gamma_1]$ is strict monoidal, and so the natural map $[\Gamma_1](n)\otimes [\Gamma_1](n)\to [\Gamma_1](2n)$ is an isomorphism, whose inverse we are using.		
		Similarly as before, it is enough to prove the isomorphism for $\Gamma_1$ free.
		In this case, the first part of Proposition \ref{prop:tensor-for-represen-functors} implies the result.

        The above argument assumes that the group $\Gamma$ is finitely generated -- if this is not the case,
        the result follows from the result for finitely generated groups and the standard fact that any group
        is a colimit of finitely generated subgroups.
	\end{proof}
	
	\begin{remark}
		In general, for such a convolution in $\hom^\Ho(F,G)$ to exist, one needs $F$ to be an \emph{oplax monoidal functor}
		and $G$ to be lax monoidal.
	\end{remark}
	
	\begin{proposition}
		\label{prop:tensor-isom-to-rep-var-of-Lie-alg}
		There is an isomorphism of commutative algebras $[k(G)]\otimes_\Ho [\cU\mathfrak{a}]\simeq k(\rep(\mathfrak{a},\mathfrak{g}))$,
		where $\mathfrak{a}$ is a Lie algebra,
		$\mathfrak{g}=\mathcal{L}ie(G)$ and $\rep(\mathfrak{a},\mathfrak{g})$ is the Lie representation variety.
	\end{proposition}
	\begin{proof}
		We will only sketch the proof.
		Since any Lie algebra $\mathfrak{a}$ is a quotient of a free Lie algebra $\cL ie(V)$,
		it is enough to prove the isomorphism for free Lie algebras.
		
		If $\mathfrak{a}=\cL ie(V)$, then $\cU \mathfrak{a}\simeq TV$.
		Thus, by Lemma \ref{lem:TV-functor} we have
		\[
		[k(G)]\otimes_\Ho [TV]\simeq \bigoplus\limits_{n\geqslant 0} k(G)\otimes_\Ho \left( V^{\otimes n}\otimes_{S_n} h^{[n]}\Big/ \sim \right)
		\]
		By Lemma \ref{lem:tensor-as-coeq} and Proposition \ref{prop:tensor-for-represen-functors},
		the latter is isomorphic to $\bigoplus_{n\geqslant 0}\left( k(G)^{\otimes n}\otimes_{S_n} V^{\otimes n}\right)/\sim$.
		The relations imply that $1\otimes v=0$ for any $v\in V^{\otimes n}$, $1\in k(G)^{\otimes n}$.
		Moreover, if we denote $I:=\ker(k(G)\to k)$ the kernel of the counit map,
		for any $f\in I^2\subset k(G)$, $f\otimes v=0$ in the quotient.
		Indeed, it is enough to see it for the element $X$ of the form $X=(f_1-f_1(e))(f_2-f_2(e))\otimes_{\Ho} v$.
		We have
		\begin{align*}
		X & \sim \big(\left(f_1-f_1(e)\right)\otimes \left(f_2 - f_2(e)\right)\big) \otimes_{\Ho} \Delta(v) \sim
		\big(f_1\otimes f_2 \big)\otimes_{\Ho} \big( (\id-\eta\varepsilon)\otimes(\id-\eta\varepsilon)\circ \Delta (v) \big) \\
        &\sim \big(f_1\otimes f_2 \big) \otimes_{\Ho} \big((\id-\eta\varepsilon)\otimes(\id-\eta\varepsilon)(1 \otimes v + v \otimes 1)\big) = 0
		\end{align*}
		
		Thus, the quotient is isomorphic to $\bigoplus\limits_{n\geqslant 0}(I/I^2)^{\otimes n}\otimes_{S_n} V^{\otimes n}$,
		which is isomorphic to $Sym(\mathfrak{g}^{\ast}\otimes V)\simeq k(\rep(\cL ie(V),\mathfrak{g}))$ since
		$\mathfrak{g}^\ast\simeq I/I^2$. This finishes the proof.
	\end{proof}

	We say that a Hopf algebra $H$ \emph{acts} on a Hopf algebra $A$ if $A$ has an algebra action $\mu\colon H\otimes A\to A$ of $H$,
	denoted $h\cdot a$, such that $h\cdot(ab)=(h'\cdot a)(h''\cdot b)$ and $h\cdot 1=\varepsilon(h)1$.
	The \emph{invariants} of an action of $H$ on $H'$ is by definition $\ker(\mu-\mu\circ \eta)$.
    An action of a cocommutative Hopf algebra $H$ on a Hopf algebra $H'$
    extends using the comultiplication of $H$ to an action of $H$ on the functor $[H']$.

    Starting with the action of $H$ on itself by conjugation $h \cdot h_1 = h' h_1 S(h'')$,
    we get the conjugation action of $H$ on $H^{\otimes n}$ defined in \cite{CK}.
    We can dualize this to get a coaction of a \emph{commutative} Hopf algebra $B$
    on the functor $[B]\colon \Ho^{op}\to \cV ect$. This, in turn, induces a coaction of $B$
    on the tensor product $[B]\otimes_\Ho [H]$ for any cocommutative Hopf algebra $H$.
	
	In particular, if $B=k(G)$ is the coordinate ring of an affine algebraic group and $H=k[\Gamma]$,
	this coaction is dual to the usual action of $G$ on the representation variety $\rep(\Gamma,G)$.
	If $H=\cU \mathfrak{a}$, the coaction is dual to the action of $G$ on $\rep(\mathfrak{a},\mathfrak{g})$.

		\begin{corollary}
			There is an isomorphism of commutative algebras $[k(G)]^G\otimes_\Ho [\Gamma]\simeq k(\cX(\Gamma,G))$,
			where $\cX(\Gamma,G)$ is the $G$-character variety of a discrete group $\Gamma$.
		\end{corollary}
		\begin{proof}
			There is an action of $G$ on the functor $k(G)$ induced by the conjugation.
			This action induces an action on the tensor product $[k(G)]\otimes_\Ho [\Gamma]$,
			which is the same as the action of $G$ by conjugation on the representation variety.
			By definition, the ring of functions on the character variety is the invariants of this action.
			It is trivial to see that taking invariants commutes with tensor products over $\Ho$.
		\end{proof}
    Applying the same argument to Proposition~\ref{prop:tensor-isom-to-rep-var-of-Lie-alg} gives us:
	\begin{corollary}
		There is a natural isomorphism $[k(G)]^G\otimes_\Ho [\cU\mathfrak{a}]\simeq \cX(\mathfrak{a},\mathfrak{g})$.
	\end{corollary}

    \begin{remark}
        Similarly we can define the product $k[(G)]\otimes_\Ho [\Gamma]^\Gamma$,
        however but it is not isomorphic to the character variety in general.
        For example, if $\Gamma$ is an \emph{abelian} group, $[\Gamma]^\Gamma = [\Gamma]$ and so
        $[k(G)]\otimes_\Ho [\Gamma]^\Gamma\simeq \rep(\Gamma, G) $ which is not isomorphic to $\cX(\Gamma,G)$.
        However, both $[k(G)]^G\otimes_\Ho [\Gamma]$ and $[k(G)]\otimes_\Ho [\Gamma]^\Gamma$  contain subscheme $\cX^{den}(\Gamma, G)$ of $\cX(\Gamma, G)$ corresponding
        to the representations with Zariski dense image.
    \end{remark}
	
	\medskip

    \begin{remark}
        Dually, if $\Gamma_1$ and $\Gamma_2$ are discrete groups, there is an isomorphism of algebras
        \[
        \hom^\Ho([\Gamma_1],[\Gamma_2]^{\Gamma_2})\simeq k[ \cX(\Gamma_1,\Gamma_2) ],
        \]
        where as in Theorem \ref{thm:rep-var-groups-as-tensor-prod},
        $k[\cX(\Gamma_1,\Gamma_2)]$ denotes the algebra of functions with finite support on the set $\cX(\Gamma_1,\Gamma_2)$.
    \end{remark}
    \begin{remark}
    	We can add that $\hom^{\Ho} \left([\cU\mathfrak{g}] , [\Gamma]\right)$ and $\hom^{\Ho} \left([\cU\mathfrak{g}]^{G}, [\Gamma]\right)$
    	are isomorphic as vector spaces to the algebras of differential operators on the representation/character variety supported at (near) the trivial representation.
    \end{remark}
	\begin{remark}
        Theorem~\ref{thm:rep-var-groups-as-tensor-prod} essentially appears in~\cite{MT}. However, Massuyeau--Turaev do not work with tensor products over a category and instead express it in terms of quotient of the tensor algebra over the product.
        In fact, for any cocommutative Hopf algebra $A$ and a commutative Hopf algebra $B$,
        there is a natural isomorphism of commutative algebras $[B]\otimes_\Ho [A]\simeq A_B$, where $A_B$ is the $B$-representation
        algebra of $A$ of Massuyeau--Turaev, see \cite[Section 3]{MT}.
        Thus, for an indirect proof of the above Theorem one can combine
		this isomorphism with \cite[Example 3.3]{MT}.
	\end{remark}
    \begin{remark}
        The coaction of $B$ on $[B]\otimes_\Ho[H]$ coincides with the coaction of
        $B$ on the Massuyeau--Turaev $B$-representation algebra of $H$, see \cite[Lemma 3.3]{MT}.
    \end{remark}

	\section{Remarks on representation homology and noncommutative sets}
	\label{sec:remark}
	
    Representation homology was first defined in \cite{BKR}. For a given vector space $V$,
    associating to an associative algebra $A$ the commutative algebra of functions on the corresponding
    representation variety $\rep(A,V)$ defines a functor $\cA lg_k\to \cC om\cA lg_k$.
    It can be extended to a functor $(-)_V\colon \mathcal{DGA}_k\to \mathcal{DGCA}_k$ from the category of
    differential graded (dg) algebras to the category of dg commutative algebras.

    By \cite[Theorem 2.2]{BKR}, there exists a \emph{derived functor} $\dL(-)_V\colon \cH o(\mathcal{DGA})_k\to \cH o(\mathcal{GDCA}_k)$
    between the corresponding homotopy categories (in the sense of Quillen's model categories),
    and therefore one can define \emph{representation homology}
    of an algebra $A$ with coefficients in $V$ as the homology $\h_\bullet(A,V)=\h_\bullet(\dL(A)_V)$.
    These homology groups are computed as the homology of the (commutative) dg algebra $(R)_V$,
    where $R\sonto A$ is a semi-free dg algebra quasi-isomorphic to $A$.

    It is possible to generalize this construction of representation homology to dg algebras
    over any binary quadratic operad, see \cite[Appendix A]{BFPRW}. In particular,
    one can define \emph{Lie representation homology} groups $\h_\bullet(\mathfrak{a}, \mathfrak{g})$ for any
    (dg) Lie algebras $\mathfrak{a}$ and $\mathfrak{g}$.

    The representation homology for groups was obtained in the recent paper \cite{BRY}
    in terms of a non-abelian derived functor from the category of simplicial groups $s\cG r$ to
    the category $s\cC om\cA lg_k$ of simplicial commutative algebras, see \cite[Section 4.2]{BRY}.

    The isomorphisms $k(G)\otimes_\Ho [\Gamma]\simeq k(\rep(\Gamma,G))$ of Theorem \ref{thm:rep-var-groups-as-tensor-prod}
    and Proposition \ref{prop:tensor-isom-to-rep-var-of-Lie-alg} allow to define representation
    homology of Lie algebras and groups in a unified way, using the usual abelian (as opposed to model-theoretic)
    derived functors. For the detailed discussion on the subject see \cite[Section 4.2]{BRY}
    and in particular Corollary 4.2.
	
	\medskip
		
	The Sweedler product of Anel-Joyal (see \cite[Section 3.4]{AJ})
	can also be interpreted as a tensor product over a the category of
	so-called \emph{non-commutative sets} $\cF(as)$. This is the category with $\ob(\cF(as))=\dN=\{[0],[1],\dots \}$.
	A morphism from $[m]$ to $[n]$ is a map of sets $f\colon [m]=\{1,\dots,m \}\to \{ 1,\dots,n\}=[n]$
	with a total ordering on $f^{-1}(i)$ for all $i\in [n]$.
	This category is the PROP of unital associative algebras, i.e. the category of strict monoidal functors $\cF(as)\to \cV ect$
	is equivalent to the category of unital associative algebras, see \cite[Section 3]{Pir02}.
	If $A$ is an associative unital algebra, and $C$ is a coassociative counital coalgebra, then
	$[C]\otimes_{\cF(as)}[A]\simeq C\vartriangleright A:=T(C\otimes A)/\langle c\otimes ab-(c^{(1)}\otimes a)\otimes (c^{(2)}\otimes b),c\otimes 1-\varepsilon(c)\cdot1 \rangle$ is the Sweedler product of \cite[Section 3.4]{AJ}.
	
	In particular, if $C=\End(V)^\ast$ is linear dual coalgebra of the algebra of endomorphisms of a finite dimensional vector space $V$,
	then the abelianization $([C]\otimes_{\cF(as)}[A])_{ab}$ is isomorphic to the algebra $k(\rep(A,V))$ of regular functions
	in the representation scheme of the algebra $A$.

	\section{Deformations of representation varieties for surface groups}
	\label{sec:deform}
	
	In this section we will sketch how the tensor product construction allows to deform the
	representation and character varieties of surface groups. For all the details of the constructions and proofs
	see \cite{KP}.
	
	From now on we assume $\Gamma=\pi_1(\Sigma_g)$ is the fundamental group of a closed oriented surface of genus~$g$.
	The group $\Gamma$ has a presentation $\Gamma=\langle a_1,\dots,a_g,b_1,\dots,b_g \mid [a_1,b_1]\dots[a_g,b_g]=1 \rangle$.
	
	The idea is to use the isomorphisms $k(G)\otimes_\Ho [\Gamma]\simeq k(\rep(\Gamma,G))$
	and $k(G)^G\otimes_\Ho [\Gamma]\simeq k(\cX(\Gamma,G))$, and ``deform'' the three objects on the right hand side:
	the category $\Ho$, the functor $[\Gamma]$, and the functors $k(G)$ and $k(G)^G$.

    \medskip

    We define category $\Ro$, which is a variation of the category of ribbons, as follows.
    In the simplest case the objects of the category $\Ro$ are disks with even number of disjoint embedded intervals plus additional combinatorial data, including in particular a fixed-point free involution on the set of the embedded intervals.
    Informally, we think about objects in $\Ro$ as disks with several ribbons attracted to it at the embedded intervals.
    The involution describes which intervals are endpoints of the same ribbon.

    The morphisms in this category are ``ribbon diagrams'' in the cylinder $D \times [0,1]$,
    constructed using five types of operations.
    The first one is just extending the ribbons vertically in the cylinder,
    i.e., continuously varying the embedded intervals.
    The second one is creating
    a new ribbon i.e., adding a ``cup;'' the third is connecting the ends of two different ribbons together (adding a ``cap'');
    terminating a ribbon (``stump''); and the last one is called ``split,''
    where a ribbon gets divided into two parallel ones by cutting it lengthwise. Pictorially these ``special'' morphisms can be visualized as follows	
	
	
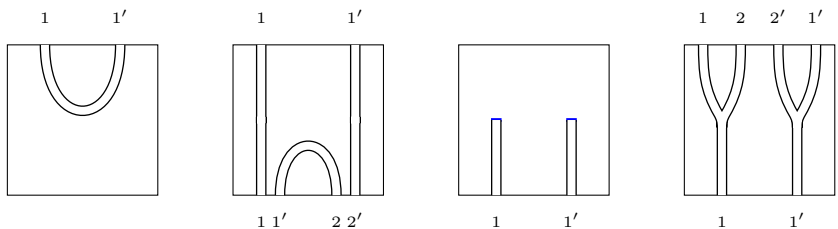
\begin{center}
		\begin{tikzpicture}[node distance=0.15,scale=0.5]
		
	\begin{scope}
	\draw (0,0) rectangle (4,4);
\begin {scope}[shift={(1,4)}]
\coordinate (La) at (0,0) ;
\coordinate (La-f) at (-90:0.2);
\coordinate (La-b) at (90:0.2);
\coordinate (La-l) at (0:0.2);
\coordinate (La-r) at (-180:0.2);
\coordinate (La-t) at (0,0);
\end {scope}
\begin {scope}[shift={(3,4)}]
\coordinate (LA) at (0,0) ;
\coordinate (LA-f) at (-90:0.2);
\coordinate (LA-b) at (90:0.2);
\coordinate (LA-l) at (0:0.2);
\coordinate (LA-r) at (-180:0.2);
\coordinate (LA-t) at (0,0);
\end {scope}
\node [above=of La]{\tiny $1$};
\node [above=of LA]{\tiny $1\pprime $};
\draw[line width=0.5pt,double,double distance=3pt]
(La)[out=-90]
to[in=-90,looseness=3] (LA)
;
	\end{scope}
	
		\begin{scope}[shift={(6,0)}]
		\draw (0,0) rectangle (4,4);
\begin {scope}[shift={(0.75,4)}]
\coordinate (LaU) at (0,0) ;
\coordinate (LaU-f) at (-90:0.2);
\coordinate (LaU-b) at (90:0.2);
\coordinate (LaU-l) at (0:0.2);
\coordinate (LaU-r) at (-180:0.2);
\coordinate (LaU-t) at (0,0);
\end {scope}
\begin {scope}[shift={(0.75,0)}]
\coordinate (LaD) at (0,0) ;
\coordinate (LaD-f) at (90:0.2);
\coordinate (LaD-b) at (270:0.2);
\coordinate (LaD-l) at (180:0.2);
\coordinate (LaD-r) at (0:0.2);
\coordinate (LaD-t) at (0,0);
\end {scope}
\begin {scope}[shift={(1.25,0)}]
\coordinate (LAD) at (0,0) ;
\coordinate (LAD-f) at (90:0.2);
\coordinate (LAD-b) at (270:0.2);
\coordinate (LAD-l) at (180:0.2);
\coordinate (LAD-r) at (0:0.2);
\coordinate (LAD-t) at (0,0);
\end {scope}
\begin {scope}[shift={(2.75,0)}]
\coordinate (LbD) at (0,0) ;
\coordinate (LbD-f) at (90:0.2);
\coordinate (LbD-b) at (270:0.2);
\coordinate (LbD-l) at (180:0.2);
\coordinate (LbD-r) at (0:0.2);
\coordinate (LbD-t) at (0,0);
\end {scope}
\begin {scope}[shift={(3.25,0)}]
\coordinate (LBD) at (0,0) ;
\coordinate (LBD-f) at (90:0.2);
\coordinate (LBD-b) at (270:0.2);
\coordinate (LBD-l) at (180:0.2);
\coordinate (LBD-r) at (0:0.2);
\coordinate (LBD-t) at (0,0);
\end {scope}
\begin {scope}[shift={(3.25,4)}]
\coordinate (LAU) at (0,0) ;
\coordinate (LAU-f) at (-90:0.2);
\coordinate (LAU-b) at (90:0.2);
\coordinate (LAU-l) at (0:0.2);
\coordinate (LAU-r) at (-180:0.2);
\coordinate (LAU-t) at (0,0);
\end {scope}
\node [above=of LaU]{\tiny $1$};
\node [below=of LaD]{\tiny $1$};
\node [below=of LbD]{\tiny $2$};
\node [below=of LBD]{\tiny $2\pprime $};
\node [below=of LAD]{\tiny $1\pprime $};
\node [above=of LAU]{\tiny $1\pprime $};

\draw [line width=0.5pt,double,double distance=3pt]
(LaD)[out=90]
to[in=-90,looseness=3] (LaU)
(LBD)[out=90]
to[in=-90,looseness=3] (LAU)
(LAD)[out=90]
to[in=90,looseness=3] (LbD)
;
		\end{scope}

		\begin{scope}[shift={(12,0)}]
		\draw (0,0) rectangle (4,4);
\begin {scope}[shift={(1,0)}]
\coordinate (LaD) at (0,0) ;
\coordinate (LaD-f) at (90:0.05);
\coordinate (LaD-b) at (270:0.05);
\coordinate (LaD-l) at (180:0.05);
\coordinate (LaD-r) at (0:0.05);
\coordinate (LaD-t) at (0,0);
\end {scope}
\begin {scope}[shift={(3,0)}]
\coordinate (LAD) at (0,0) ;
\coordinate (LAD-f) at (90:0.05);
\coordinate (LAD-b) at (270:0.05);
\coordinate (LAD-l) at (180:0.05);
\coordinate (LAD-r) at (0:0.05);
\coordinate (LAD-t) at (0,0);
\end {scope}
\node [below=of LaD]{\tiny $1$};
\node [below=of LAD]{\tiny $1\pprime $};
\begin {scope}[shift={(1,2)}]
\coordinate (LaM) at (0,0) ;
\coordinate (LaM-f) at (-90:0.05);
\coordinate (LaM-b) at (90:0.05);
\coordinate (LaM-l) at (0:0.05);
\coordinate (LaM-r) at (-180:0.05);
\coordinate (LaM-t) at (0,0);
\end {scope}
\begin {scope}[shift={(3,2)}]
\coordinate (LAM) at (0,0) ;
\coordinate (LAM-f) at (-90:0.05);
\coordinate (LAM-b) at (90:0.05);
\coordinate (LAM-l) at (0:0.05);
\coordinate (LAM-r) at (-180:0.05);
\coordinate (LAM-t) at (0,0);
\end {scope}
\begin {scope}
\clip (LaM-b) circle (0.2);
\draw [u](LaM-b) -- (LaM-f);
\end {scope}
\begin {scope}
\clip (LAM-b) circle (0.2);
\draw [u](LAM-b) -- (LAM-f);
\end {scope}
\draw[line width=0.5pt,double,double distance=3pt]
(LaD)[out=90]
to[in=-90,] (LaM)
;
\draw[line width=0.5pt,double,double distance=3pt]
(LAD)[out=90]
to[in=-90,] (LAM)
;
		\end{scope}

		\begin{scope}[shift={(18,0)}]
		\draw (0,0) rectangle (4,4);
\begin {scope}[shift={(1,0)}]
\coordinate (LaD) at (0,0) ;
\coordinate (LaD-f) at (90:0.2);
\coordinate (LaD-b) at (270:0.2);
\coordinate (LaD-l) at (180:0.2);
\coordinate (LaD-r) at (0:0.2);
\coordinate (LaD-t) at (0,0);
\end {scope}
\begin {scope}[shift={(3,0)}]
\coordinate (LAD) at (0,0) ;
\coordinate (LAD-f) at (90:0.2);
\coordinate (LAD-b) at (270:0.2);
\coordinate (LAD-l) at (180:0.2);
\coordinate (LAD-r) at (0:0.2);
\coordinate (LAD-t) at (0,0);
\end {scope}
\node [below=of LaD]{\tiny $1$};
\node [below=of LAD]{\tiny $1\pprime $};
\begin {scope}[shift={(0.5,4)}]
\coordinate (LaU) at (0,0) ;
\coordinate (LaU-f) at (-90:0.2);
\coordinate (LaU-b) at (90:0.2);
\coordinate (LaU-l) at (0:0.2);
\coordinate (LaU-r) at (-180:0.2);
\coordinate (LaU-t) at (0,0);
\end {scope}
\begin {scope}[shift={(1.5,4)}]
\coordinate (LbU) at (0,0) ;
\coordinate (LbU-f) at (-90:0.2);
\coordinate (LbU-b) at (90:0.2);
\coordinate (LbU-l) at (0:0.2);
\coordinate (LbU-r) at (-180:0.2);
\coordinate (LbU-t) at (0,0);
\end {scope}
\begin {scope}[shift={(2.5,4)}]
\coordinate (LBU) at (0,0) ;
\coordinate (LBU-f) at (-90:0.2);
\coordinate (LBU-b) at (90:0.2);
\coordinate (LBU-l) at (0:0.2);
\coordinate (LBU-r) at (-180:0.2);
\coordinate (LBU-t) at (0,0);
\end {scope}
\begin {scope}[shift={(3.5,4)}]
\coordinate (LAU) at (0,0) ;
\coordinate (LAU-f) at (-90:0.2);
\coordinate (LAU-b) at (90:0.2);
\coordinate (LAU-l) at (0:0.2);
\coordinate (LAU-r) at (-180:0.2);
\coordinate (LAU-t) at (0,0);
\end {scope}
\node [above=of LaU]{\tiny $1$};
\node [above=of LbU]{\tiny $2$};
\node [above=of LBU]{\tiny $2\pprime $};
\node [above=of LAU]{\tiny $1\pprime $};
\begin {scope}[shift={(1,2)}]
\coordinate (Sy) at (0,0);
\coordinate (Sy-f) at (90:0.2);
\coordinate (Sy-l) at (120:0.2);
\coordinate (Sy-r) at (60:0.2);
\coordinate (Sy-b) at (270:0.2);
\coordinate (Sy-t) at (270:0.2);
\end {scope}
\begin {scope}[shift={(3,2)}]
\coordinate (SY) at (0,0);
\coordinate (SY-f) at (90:0.2);
\coordinate (SY-l) at (120:0.2);
\coordinate (SY-r) at (60:0.2);
\coordinate (SY-b) at (270:0.2);
\coordinate (SY-t) at (270:0.2);
\end {scope}
\draw[line width=0.5pt,double,double distance=3pt]
(LaD)[out=90]
to[in=90,] (Sy)
(Sy-b)[out=90] to[in=240] (Sy-r)[out=60]
to[in=-90,] (LbU)
(Sy-b)[out=90] to[in=300] (Sy-l)[out=120]
to[in=-90,] (LaU)
;
\draw[line width=0.5pt,double,double distance=3pt]
(LAD)[out=90]
to[in=90,] (SY)
(SY-b)[out=90] to[in=300] (SY-l)[out=120]
to[in=-90,] (LBU)
(SY-b)[out=90] to[in=240] (SY-r)[out=60]
to[in=-90,] (LAU)
;
		\end{scope}
		\end{tikzpicture}
	\end{center}

    These basic morphisms satisfy many relations: two morphism are the same if the corresponding diagrams are homotopic;
    adding a cup followed by a cap is equivalent to a trivial morphism;
    splits and stumps interact as expected; and so on.

    The category $\Ro$ has a full subcategory $\Ro_0$ with objects indexed by the integers,
    where the $n$-th object corresponds to a disk with $2n$ embedded intervals in a standard position.
    There is a projection of the set of morphism in $\Ro_0$ to the morphisms in $\Ho$ by identifying ``ribbon diagrams'' which can be obtained by crossings. This explains why the category $\Ro_0$
    can be viewed as deformation of $\Ho$.

    \smallskip

    For any $3$ manifold $M$ with a disk on the boundary we can define a functor $[M]$ from $\Ro$ to $\cS ets$,
    a disk with embedded intervals is sent to the set of equivalence classes of ribbons inside $M$
    which terminate at these intervals. This restricts to a functor $[M]$ (some time called quantum fundamental group~\cite{Hab})
    from $\Ro_0$ which can be viewed as a deformation of the functors $[\pi_1(M)]$,
    because modulo crossings, the set of $n$ ribbons in $M$ can be identified with the element in $\pi_1(M)^{\times n}$.
    It is possible to obtain presentation of the functor $[M]$ using the geometry of the manifold.

    \smallskip

    It is a classical result that ribbon Hopf algebras (like quantum groups) lead to functors from ribbon categories to vector spaces,
    and this is the case in our situation. We are mainly interested in the dual case of coribbon algebra.
    To any semi-simple affine algebraic group scheme $G$ one can associate a coribbon Hopf algebra $k_q(G)$
    which gives a pair of functors from $\Ro_0$, one sending $[n]$ to $k_q(G)^{\otimes n}$
    and the other to a the space of invariants under suitable action of $k_q(G)$ on it.

    \medskip

    This gives us a deformation of the tensor products $k(G)\otimes_\Ho [\pi_1(M)]$ and $k(G)^G \otimes_\Ho [\pi_1(M)]$,
    which are algebras of functions on the representation and character variety, respectively.
    The presentation of $[M]$ makes it possible to compute these deformations using an analog of Lemma~\ref{lem:tensor-as-coeq}.
    For example when $[M]$ is a thickened torus and $G=\dC^\times$, the resulting
    deformation is isomorphic to the Weyl algebra $A_q\simeq k\langle x,p \rangle/([x,p]=q)$.

    In general this deformation of the character variety is not an algebra,
    since it is not possible to put a lax moinoidal structure on the functor $[M]$.
    However, in the case when $M=\Sigma\times [0,1]$ is a thickened surface, the functor $[M]$
    has a natural monoidal structure induced by stacking the two copies of $\Sigma\times [0,1]$ together.%
    \footnote{This requires putting a monoidal structure on $\Ro$, which requires replacing the disk with an annulus,
    	and other technical details which we are suppressing.}

    \medskip

    When $q$ is root of unity there are several forms of the quantum group $U_q(g)$ and $k_q(G)$.
    By taking the small quantum group, one obtains a finite dimensional space -- one of the main results in \cite{KP}
    shows that in this case the analog of the representation variety is a finite dimensional algebra
    which is closely related to skein model of quantum representation of the mapping class group.
    We want to stress that in our case the action of the mapping class group in these algebra follows
    immediately from the functionality of the construction.

	\appendix
	
	\section{Background material}
	
	\subsection{Axioms for Hopf algebras}
	\label{subsec:axioms-Hopf-alg}
	
	We say that a vector space $H$ over a field $k$ has a structure of a \emph{Hopf algebra} if there are maps
	$\eta\colon k\to H$ called \emph{unit}, \emph{multiplication map} $\mu\colon H^{\otimes 2}\to H$,
	\emph{counit map} $\varepsilon\colon H\to k$, \emph{comultiplication} $\Delta \colon H\to H^{\otimes 2}$ and
	the antipode map $S\colon H\to H$ satisfying the following relations:
		\begin{enumerate}
			\item $\mu\circ (\mu\otimes \id)=\mu\circ (\id \otimes \mu)$ (associativity);
			\item $\mu\circ (\eta\otimes \id)=\mu\circ (\id\otimes \eta)=\id$ (unit axiom);
			\item $(\id\otimes \Delta)\circ \Delta = (\Delta\otimes \id)\circ \Delta$ (coassociativity);
			\item $(\id\otimes \varepsilon)\circ \Delta = (\varepsilon\circ \id)\circ \Delta =\id$ (counit axiom);
			\item $\Delta\circ \mu=(\mu\otimes \mu)\circ (\id\otimes \tau\otimes \id)\circ (\Delta\otimes \Delta)$ (compatibility of $\mu$ and $\Delta$);
			\item $\mu\circ (\id\otimes S)\circ \Delta=\eta\circ \varepsilon=\mu\circ (S\otimes \id) \Delta$ (the antipode axiom);
			\item $\Delta\circ S= (S\otimes S)\circ \tau\circ \Delta$ (compatibility of $S$ and $\Delta$),
			where $\tau\colon H^{\otimes 2}\to H^{\otimes 2}$ is defined by $\tau(x\otimes y)=y\otimes x$;
			\item $S\circ \mu=\mu\circ \tau\circ (S\otimes S)$ (compatibility of $S$ and $\mu$).	
		\end{enumerate}
		
		We say that $H$ is \emph{cocommutative} if in addition the following relation holds:
		\begin{enumerate}
			\item[9.] $\tau\circ \Delta=\Delta$ (cocommutativity).
		\end{enumerate}
        The cocommutativity implies that the antipode is of order $2$:
		\begin{enumerate}
			\item[10.] $S\circ S=\id$.
		\end{enumerate}
		
    \subsection{Basic category theory}
	
    \begin{definition}
    	A \emph{monoidal category} $\cC$ is a category equipped with
    	\begin{enumerate}
    		\item a functor $\otimes \colon \cC\times \cC\to \cC$ called the \emph{tensor product};
    		\item a \emph{unit object} $1_\cC\in\ob(\cC)$;
    		\item a natural isomorphism $\alpha\colon \left( (-)\otimes (-) \right)\otimes (-)\sto (-) \otimes \left( (-)\otimes (-) \right)$
    	called the \emph{associator};
	    	\item two natural isomorphisms $1_\cC\otimes (-)\sto \id_\cC$ and $(-)\otimes 1_\cC\sto \id_\cC$ called
	    	\emph{left} and \emph{right unitor}, respectively;
    	\end{enumerate}
    	satisfying some natural compatibility conditions, see \cite[Sect. VII.1]{ML} for details.
    \end{definition}

    \begin{definition}
    	A monoidal category $\cC$ is called \emph{braided} if it is equipped with a natural isomorphism
    	$\tau_{x,y}\colon x\otimes y\sto y\otimes x$, $x,y\in\ob(\cC)$ called the \emph{braiding}.
    	If $\tau_{y,x}\circ\tau_{x,y}=\id_{x\otimes y}$, the category $\cC$ is called \emph{symmetric monoidal}.
    	The braiding should satisfy certain natural conditions.	We refer the reader to \cite[Sect. XI]{ML} for all the details.
    \end{definition}

	\begin{definition}
		A \textbf{lax monoidal functor} $F\colon \cC\to \cD$ between monoidal categories $\cC,\cD$ is a triple $(F,\varepsilon_F,\varphi)$ consisting of a usual functor $F\colon \cC\to \cD$, a morphism
		$\varepsilon_F\colon 1_\cD\to F(1_\cC)$ and a natural transformation
		$F(c)\otimes F(c')\stackrel{\varphi_{cc'}}{\longrightarrow} F(c\otimes c')$ for any $c,c'\in \ob(\cC)$
		satisfying certain \emph{associativity} and \emph{unitality} conditions, see \cite[Section XI.2]{ML} for details.
		The naturality means that for any $f\colon c\to c'$ morphism in $\cC$, the following diagram commutes:
		\begin{equation}
		\label{eq:nat-cond-lax-mon-str}
		\xymatrix{
			F(c)\otimes F(c'') \ar[d]_-{F(f)\otimes \id_{c''}} \ar[r]^-{\varphi_{cc''}} & F(c\otimes c'') \ar[d]^-{F(f\otimes\id_{c''})}\\
			F(c')\otimes F(c'') \ar[r]^-{\varphi_{c'c''}} & F(c'\otimes c'')
		}
		\end{equation}
	\end{definition}

	If $X,Y\in \ob(\cC)$ are objects of a category $\cC$, and $\varphi_i\colon X\to Y$ is a collection of morphisms
	indexed by $i\in I$, the \emph{coequalizer} of this collection is an object $Z\in \cC$ with a morphism $\alpha\colon Y\to Z$
	satisfying the following universal property. For any $A\in\ob(\cC)$ with a morphism $\beta\colon Y\to A$
	satisfying $\beta\circ \varphi_i=\beta\circ \varphi_j$ there exists a unique morphism $\overline{\beta}\colon Z\to A$ such that
	$\beta=\overline{\beta}\circ \alpha$.

\medskip

    	Let $\cD$ be a small category which has equalizers and coequalizers, and it also has coproducts. Then for any set $S$ and any object $d \in \cD$ we have a well defined object
    	$d\cdot S =\coprod\limits_{S}d$ which is the coproduct of $S$ copies of $d$.
    	
    	Let $F\colon \cC\to \cS ets$ and $G\colon \cC^{op}\to \cD$ be two functors, and assume that category $\cD$
	is cocomplete. Define the tensor product $G\otimes_\cC F$ as
	\[
	G\otimes_\cC F:=\operatorname{coeq}\left[ \coprod\limits_{c\to c'} G(c')\cdot F(c) \rightrightarrows \coprod\limits_{c} G(c)\cdot F(c) \right]
	\]
	
		\begin{proposition}
			\label{prop:alg-structure-on-tensor}
			Let $F\colon \cC\to \cS ets$ and $G\colon \cC^{op}\to \cD$ be two lax monoidal functors between two monoidal categories.
			Assume that $\cD$ is also \emph{braided} and \emph{disctributive}, and that the maps $\varphi_{cc'}$ and $\psi_{cc'}$ respect the braiding.
			Then $G\otimes_\cC F$ has a natural structure of an algebra object in $\cD$.
			Moreover, if $\cD$ is \emph{symmetric} monoidal, then the algebra $G\otimes_\cC F$ is commutative.
		\end{proposition}
		\begin{proof}
			Define
			\[
			(\psi_{cc'}\otimes \varphi_{cc'})\circ (\id\otimes \tau \otimes \id)\colon G(c)\otimes F(c)\otimes G(c')\otimes F(c')\to G(c\otimes c')\otimes F(c\otimes c')
			\]
			using the braiding operator $\tau=\tau_{F(c),G(c')}$ on the symmetric monoidal category $\cD$.
			The naturality conditions \eqref{eq:nat-cond-lax-mon-str} imply that these maps induce a well-defined multiplication map
			on $G\otimes_\cC F$. The unit $1_\cD\to G\otimes_\cC F$ is induced by natural map
			\[
			1_\cD\simeq 1_\cD\otimes 1_\cD \stackrel{\varepsilon_G\otimes \varepsilon_F}{\longrightarrow} G(1_\cC)\otimes F(1_\cC)
			\]
			The associativity and unitality for $G\otimes_\cC F$ follow from the associativity and unitality
			axioms for lax monoidal functors.			
		\end{proof}

	\begin{lemma}
		\label{lem:tensor-with-corepres-functor}
		If $F=h^x$ for some $x\in\ob\cC$, then $G\otimes_\cC h^{x}\simeq G(x)$.
	\end{lemma}
	\begin{proof}
		In this case $G\otimes_\cC h^{x}$ is the coequalizer of the diagram $\xymatrix{
			\coprod\limits_{x\to c\to c'} G(c') \ar@<+2pt>[r]^-{\alpha} \ar@<-2pt>[r]_-{\beta} &
			\coprod\limits_{x\to c} G(c)
		}$ where $\alpha$ is $\id\colon G(c')_{x\to c\to c'}\to G(c')_{x\to c'}$ sending the component indexed by $x\to c\to c'$
		to the one indexed by the composition $x\to c'$.
		The map $\beta$ is $\xymatrix{G(c')_{x\to c\to c'}\ar[rr]^-{G(c\to c')} & & G(c)_{x\to c}}$.
		
		There is a natural map
		\begin{equation}
		\label{eq:coeq-map-repres-tensor-prod}
		\varphi\colon \coprod\limits_{x\to c}G(c) \to G(x)
		\end{equation}
		which on the component $G(c)$ indexed by a morphism $x\to c$ is given by $G(x\to c)\colon G(c)\to G(x)$.
		To prove the desired isomorphism, we need to check that $\varphi$ satisfies the coequalizer universal property.
		In other words, we need to check that for any $\psi\colon \coprod\limits_{x\to c}G(c) \to Y $
		there exists unique $\overline{\psi}$ making the following diagram commute:
		\[
		\xymatrix{
			\coprod\limits_{x\to c\to c'} G(c') \ar@<+2pt>[r] \ar@<-2pt>[r] &
			\coprod\limits_{x\to c} G(c) \ar[d]_-{\psi} \ar[r]^-{\varphi} & G(x) \ar@{-->}[dl]^-{\exists ! \overline{\psi}}\\
			& Y
		}
		\]
		We define $\overline{\psi}$ to be the restriction of $\psi$ on the component $G(c)$ indexed by $\id_x\colon x\to x$.
		It is a direct check that the resulting diagram commutes.
	\end{proof}

\begin{lemma}
	\label{lem:corepres-funct-is-lax-monoid}
	Suppose that $x\in \ob(\cC)$ is a \emph{counital coalgbra object} in $\cC$, i.e. there are morphisms $x\to x\otimes x$ and $x\to 1_\cC$ satisfying the usual compatibility axioms. Then $F:=h^x\colon \cC\to \cS ets$ is a lax monoidal functor,
	where the monoidal structure on $\cS ets$ is given by the product $\times$ of sets.
\end{lemma}
\begin{proof}
	The map $x\to x\otimes x$ induces the natural map $F(y)\times F(y')\to F(y\otimes y')$,
	and the map $x\to 1_\cC$ induces the natural map $1_{\cS ets}=\{\ast \}\to F(1_\cC)=\hom(x,1_\cC)$.
\end{proof}

\begin{lemma}
	\label{lem:isom-of-algebras-from-tensor}
	Let $x$ is a counital coalgebra object in $\cC$, $F=h^x\colon \cC\to \cS ets$ is the corresponding
	functor, and let $G\colon \cC^{op}\to \cD$ a lax monoidal contravariant functor to a cocomplete braided monoidal category $\cD$.
	Then the isomorphism $G\otimes_\cC h^x \simeq G(x)$ of Lemma \ref{lem:tensor-with-corepres-functor}
	is an \emph{isomorphism of algebras}.
\end{lemma}
\begin{proof}
	The functor $F=h^x$ is lax monoidal by Lemma \ref{lem:corepres-funct-is-lax-monoid}, and so
	by Proposition \ref{prop:alg-structure-on-tensor}, the object $G\otimes_\cC F\in \ob(\cD)$
	is indeed an algebra. The algebra structure on $G(x)$ is induced by $G(x)\otimes G(x) \to G(x\otimes x) \to G(x) $
	and $1_\cD\to G(1_\cC)\to G(x)$.
	
	Since the algebra structure on the tensor product $G\otimes_\cC h^x$ is induced by the algebra structure
	on $\coprod\limits_{x\to c} G(c)$, it is enough to verify that the map $\varphi$ from \eqref{eq:coeq-map-repres-tensor-prod}
	is an algebra homomorphism. This is a direct verification.
\end{proof}

	\begin{lemma}
		\label{lem:colims-in-functors}
		Assume $\cD$ has limits and colimits.
		Then the functor category $\cF un(\mathcal{C}, \mathcal{D})$ has limits and colimits
		(in particular equalizers and coequalizers), which can be computed point wise.
	\end{lemma}
	\begin{proof}
		It's a well-known fact. See, for example, \cite[Section 3.3]{Kel05}.
	\end{proof}

\medskip
    The only slightly non-trivial result we need is that tensor product over categories commute with coequalizers.

	\begin{lemma}
		\label{lem:tensor-as-coeq}
		If $\otimes$ commutes with colimits in the category $\mathcal{D}$, then the tensor product $G \otimes_\mathcal{C} F/\{\varphi_i\}$
		is the coequalizer in $\mathcal{D}$ of the maps $\overline{\varphi}_i: G(c) \otimes d   \to G \otimes_\mathcal{C} F$.
	\end{lemma}
	\begin{proof}
		The idea of the proof is simple. We will express $G \otimes_\mathcal{C} F/\{\varphi_i\}$ as a colimit in
		the functor category $Fun(I\times J,D)$ where $I$ and $J$ are certain index categories. Then we will
		use the fact that this colimit can be computed coordinate-wise (see, for example, \cite[Eq. (2.19)]{EH69}).
		When doing $\colim_I\colim_J$ we get the coequalizer of the maps
		$\overline{\varphi}_i: G(c) \otimes d \to G \otimes_\mathcal{C} F$.
		When doing $\colim_J\colim_I$ we will get $G \otimes_\mathcal{C} F/\{\varphi_i\}$.
		
		By definition, $G \otimes_\mathcal{C} F$ is the coequalizer of the diagram
		\[
		\xymatrix{
			\coprod\limits_{c'\stackrel{f}{\to} c''} G(c'')\otimes \left(F/\{\varphi_i\}\right)(c') \ar@<+3pt>[r]^-{f_\ast} \ar@<-3pt>[r]_-{f^\ast} &
			\coprod\limits_{c'''} G(c''')\otimes \left(F/\{\varphi_i\}\right)(c''')
		}
		\]
		Moreover, $\left(F/\{\varphi_i\}\right)(c')$ is the coequalizer of the diagram
		\[
		\xymatrix{
			\coprod\limits_{c\to c'} d \ar@<+3pt>[r]^-{\varphi_i} \ar@<-3pt>[r] \ar[r] & F(c')
		}
		\]
		
		Therefore, using the fact that $\otimes$ commutes with colimits, we have
		\[
					\coprod\limits_{c'\to c''} G(c'')\otimes \left(F/\{\varphi_i\}\right)(c')
					\simeq \colim_{\varphi_i} \left[ \coprod\limits_{c\to c'\to c''} G(c'')\otimes d \rightrightarrows
					\coprod\limits_{c'\to c''} G(c'')\otimes F(c') \right]
		\]
		
		Similarly, there is an isomorphism
		\[
		\coprod\limits_{c'''} G(c''')\otimes \left(F/\{\varphi_i\}\right)(c''') \simeq \colim_{\varphi_i}\left[ \coprod\limits_{c\to c'''} G(c''')\otimes d \rightrightarrows \coprod\limits_{c'''} G(c''')\otimes F(c''') \right]
		\]
		
		Next, we can realize $G(c)\otimes d$ as the coequalizer of $\coprod\limits_{c\to c'\to c''} G(c'')\otimes d \rightrightarrows \coprod\limits_{c\to c'''} G(c''')\otimes d $ where the two maps we are coequalizing are as follows. One is the identity map
		from $G(c'')\otimes d$ labeled by a chain $c\stackrel{f}{\to} c'\stackrel{g}{\to} c''$ to $G(c'')\otimes d$ labeled by the composition $c\stackrel{g\circ f}{\longrightarrow} c''$. The second map, restricted to the $G(c'')\otimes d$ labeled by $c\stackrel{f}{\to} c'\stackrel{g}{\to} c''$
		is the map $G(f)\otimes d\colon G(c'')\otimes d\to G(c')\otimes d$. Therefore, diagrammatically, we have the following picture:
		\[
		\xymatrix{
			A \ar@<+3pt>[d] \ar@<-3pt>[d] \ar@<+3pt>[r]^-{\varphi_i} \ar@<-3pt>[r] \ar[r] & B \ar@<+3pt>[d] \ar@<-3pt>[d] \ar@{-->}[r]^-{\colim} & X \ar@<+3pt>[d] \ar@<-3pt>[d]\\
			C \ar@{-->}[d]_-{\colim} \ar@<+3pt>[r]^-{\varphi_i} \ar@<-3pt>[r] \ar[r] & D \ar@{-->}[d]_-{\colim} \ar@{-->}[r]^-{\colim} & Y \ar@{-->}[d]_-{\colim}\\
			G(c)\otimes d  \ar@<+3pt>[r]^-{\varphi_i} \ar@<-3pt>[r] \ar[r] & G\otimes_\cC F \ar@{-->}[r]^-{\colim} & G \otimes_\mathcal{C} F/\{\varphi_i\}&
		}
		\]
		
		Now we have two diagram categories $J=\{\bullet \rightrightarrows \bullet \}$
		(for the coend maps) and $I=\{\bullet \rightrightarrows \bullet \}$ (for the maps $\varphi_i$). By definition,
		$G\otimes_{\cC}F/\{\varphi_i\}=\colim_J\colim_I$ of the top left $2\times 2$ square in the diagram above,
		and the coequalizer $G(c)\otimes d\rightrightarrows G\otimes_\cC F$
		is by definition $\colim_I\colim_J$ of the same $2\times 2$ square.
		Now by \cite[Eq. (2.19)]{EH69} the two colimits are isomorphic.
	\end{proof}

\bibliography{references}
\bibliographystyle{acm}

\end{document}